\documentclass[12pt, reqno]{amsart}

\usepackage[utf8]{inputenc}
\usepackage[T2A]{fontenc}
\usepackage[english]{babel}

\usepackage[margin=2.5cm]{geometry}
\usepackage{enumitem}

\usepackage{amssymb}
\usepackage{amscd}
\usepackage{graphicx}
\usepackage[noadjust]{cite}
\usepackage[usenames,dvipsnames]{xcolor}
\usepackage[all]{xy}
\usepackage{tikz-cd}

\usepackage{hyperref}
\hypersetup{pdfstartview=FitH,  linkcolor=blue,urlcolor=urlcolor, citecolor=blue, colorlinks=true}

\newcommand\ie{\textit{i.e.}\ }
\newcommand\eg{\textit{e.g.}\ }

\newtheorem{theorem}[subsection]{Theorem}
\newtheorem{subtheorem}[subsubsection]{Theorem}

\newtheorem{sublemma}[subsubsection]{Lemma}

\newtheorem{subcorollary}[subsubsection]{Corollary}

\theoremstyle{definition}
\newtheorem{definition}[subsection]{Definition}
\newtheorem{subdefinition}[subsubsection]{Definition}

\newtheorem{remark}[subsection]{Remark}
\newtheorem{subremark}[subsubsection]{Remark}

\makeatletter
\@addtoreset{subsection}{section}
\@addtoreset{equation}{section}
\@addtoreset{figure}{section}
\@addtoreset{table}{section}
\makeatother

\makeatletter
\newcommand\testshape{family=\f@family; series=\f@series; shape=\f@shape.}
\def\myemphInternal#1{\if n\f@shape%
\begingroup\itshape #1\endgroup\/%
\else\begingroup\sf\itshape #1\endgroup%
\fi}
\def\myemph{\futurelet\testchar\MaybeOptArgmyemph}
\def\MaybeOptArgmyemph{\ifx[\testchar \let\next\OptArgmyemph
                 \else \let\next\NoOptArgmyemph \fi \next}
\def\OptArgmyemph[#1]#2{\index{#1}\myemphInternal{#2}}
\def\NoOptArgmyemph#1{\myemphInternal{#1}}
\makeatother
\newcommand\id{\mathrm{id}}          % identity map
\newcommand\Int[1]{\mathrm{Int}(#1)}       % interior
\newcommand\eps{\varepsilon}

\newcommand\bR{\mathbb{R}}
\newcommand\bZ{\mathbb{Z}}

\newcommand\Homeo{\mathcal{H}}      % homeomorphisms
\newcommand\Map{\mathrm{Map}}       % maps
\newcommand\HomeoId{\Homeo_{\id}}   % homeomorphisms
\newcommand\Cr[1]{\mathcal{C}^{#1}}

\newcommand\Crm[3]{\Cr{#1}\!\left(#2,#3\right)}
\newcommand\Cont[2]{\Crm{0}{#1}{#2}}                         % space of continuous maps
\newcommand\Aman{A}
\newcommand\Bman{B}
\newcommand\Cman{C}

\newcommand\Fman{F}
\newcommand\Gman{G}

\newcommand\Kman{{\color{Blue}K}}
\newcommand\Lman{L}
\newcommand\Mman{M}

\newcommand\Tman{T}
\newcommand\Uman{U}
\newcommand\Vman{V}

\newcommand\Yman{Y}

\newcommand\dif{h}
\newcommand{\predif}{\dif_0}

\newcommand\kdif{k}

\newcommand\dComp{\Cman}

\newcommand{\Cl}[2][]{\mathrm{Cl}_{#1}\!\left({#2}\right)}
\newcommand{\strip}{S}
\newcommand{\stripSurf}{Z}
\newcommand{\preStripSurf}{\stripSurf_0}

\newcommand{\Partition}{\Delta}     % canonical foliation
\newcommand\PartitionReg{\Partition_{reg}}
\newcommand\PartitionSpec{\Partition_{spec}}
\newcommand\PartitionSing{\Partition_{sing}}

\newcommand{\leaf}{\omega}          % one leaf
\newcommand\hcl[1]{\mathrm{hcl}(#1)}  % Hausdorff closure of a point
\newcommand\hclA[2]{\mathrm{hcl}_{#2}(#1)}

\newcommand{\stInd}{{\lambda}}
\newcommand{\StInd}{\Lambda}

\newcommand{\bdGlueInd}{{\gamma}}
\newcommand{\BdGlueInd}{\Gamma}

\newcommand{\qmap}{q}
\newcommand{\pr}[1][]{p_{#1}}

\newcommand{\bdX}{X}
\newcommand{\bdY}{Y}

\newcommand{\HS}[1][\Partition]{\Homeo(\stripSurf,#1)}

\newcommand\HZS{\Homeo_{\id}(\stripSurf,\Partition)}

\newcommand\HY{\Homeo(\Yman)}
\newcommand\HZY{\Homeo_{\id}(\Yman)}

\newcommand\ahom{\psi}
\newcommand\specPoints{\mathcal{S}}

\newcommand\sat[1]{\mathrm{Sat}(#1)}

\newcommand\ptx{x}

\newcommand\Fmap{F}

\newcommand\Frestr[1]{\Fmap_{#1}}
\newcommand\Gmap{G}
\newcommand\Grestr[1]{\Gmap_{#1}}

\newcommand\tHomeoId[1]{\widetilde{\mathcal{H}}_{\id}(#1)}

\newcommand\EXP[1]{E_{#1}}
\newcommand\invHZY{\mathcal{Q}}

\newcommand\branch{special}
\newcommand\maps[2]{\Map(#1,#2)}

\newcommand\fmu{\beta}
\newcommand\flambda{\alpha}
\newcommand\bSide[2]{\partial_{#1}#2}

\newcommand\Circle{S^1}
\newcommand\ori{\mathbf{or}}

\begin{document}

\label{first_page:0}
%===============================================================================
%   Заголовок --- назва, анотація, і т. ін.
%===============================================================================

\author{Sergiy Maksymenko}
\email{maks@imath.kiev.ua}
\address{Institute of Mathematics of NAS of Ukraine, Tereshchenkivska str. 3, Kyiv, 01004, Ukraine}
%\orcid{0000-0002-0062-5188}

\author{Eugene Polulyakh}
\email{polulyah@imath.kiev.ua}
\address{Institute of Mathematics of NAS of Ukraine, Tereshchenkivska str. 3, Kyiv, 01004, Ukraine}

\title[Homeotopy groups of leaf spaces of one-dimensional foliations]{Homeotopy groups of leaf spaces of one-dimensional foliations on non-compact surfaces with non-compact leaves}

\begin{abstract}
Let $Z$ be a non-compact two-dimensional manifold obtained from a family of open strips $\mathbb{R}\times(0,1)$ with boundary intervals by gluing those strips along some pairs of their boundary intervals.
Every such strip has a natural foliation into parallel lines $\mathbb{R}\times t$, $t\in(0,1)$, and boundary intervals which gives a foliation $\Delta$ on all of $Z$.
Denote by $\mathcal{H}(Z,\Delta)$ the group of all homeomorphisms of $Z$ that maps leaves of $\Delta$ onto leaves and by $\mathcal{H}(Z/\Delta)$ the group of homeomorphisms of the space of leaves endowed with the corresponding compact open topologies.
Recently, the authors identified the \textit{homeotopy} group $\pi_0\mathcal{H}(Z,\Delta)$ with a group of automorphisms of a certain graph $G$ with additional structure which encodes the combinatorics of gluing $Z$ from strips.
That graph is in a certain sense dual to the space of leaves $Z/\Delta$.

On the other hand, for every $h\in\mathcal{H}(Z,\Delta)$ the induced permutation $k$ of leaves of $\Delta$ is in fact a homeomorphism of $Z/\Delta$ and the correspondence $h\mapsto k$ is a homomorphism $\psi:\mathcal{H}(\Delta)\to\mathcal{H}(Z/\Delta)$.
The aim of the present paper is to show that $\psi$ induces a homomorphism of the corresponding \textit{homeotopy} groups $\psi_0:\pi_0\mathcal{H}(Z,\Delta)\to\pi_0\mathcal{H}(Z/\Delta)$ which turns out to be either injective or having a kernel $\bZ_2$.
This gives a dual description of $\pi_0\mathcal{H}(Z,\Delta)$ in terms of the space of leaves.
\end{abstract}

% ======== abstract in Ukrainian ========
%
% \begin{abstract}
% Нехай некомпактний двовимірний многовид $Z$ отримано з сім'ї відкритих смуг $\mathbb{R}\times(0,1)$, що мають на межі відкриті інтервали, за допомогою склеювання цих смуг вздовж деяких пар їх межових інтервалів.
% Кожна така смуга має природне шарування, листами якого є паралельні прямі $\mathbb{R}\times t$, $t\in(0,1)$, а також межові інтервали. Це породжує шарування $\Delta$ на $Z$.
% Позначимо через $\mathcal{H}(Z,\Delta)$ групу гомеоморфізмів $Z$, які відображають листи шарування $\Delta$ на листи, а через $\mathcal{H}(Z/\Delta)$ --- групу гомеоморфізмів простору листів. Наділимо ці групи компактно-відкритими топологіями.
% Нещодавно автори ідентифікували \textit{групу гомеотопій} $\pi_0\mathcal{H}(Z,\Delta)$ з групою автоморфізмів певного графу $G$ з додатковою структурою, яка кодує комбінаторику склеювання $Z$ зі смуг.
% Вказаний граф є в деякому сенсі дуальним до простору листів $Z/\Delta$.
%
% З іншого боку, для кожного $h\in\mathcal{H}(Z,\Delta)$ індукована перестановка $k$ листів $\Delta$ є гомеоморфізмом $Z/\Delta$, а відповідність $h\mapsto k$ є гомоморфізмом $\psi:\mathcal{H}(\Delta)\to\mathcal{H}(Z/\Delta)$.
% У даній роботі показано, що $\psi$ індукує гомоморфізм відповідних \textit{груп гомеотопій}
% $\psi_0:\pi_0\mathcal{H}(Z,\Delta)\to\pi_0\mathcal{H}(Z/\Delta)$, причому він або ін'єктивний, або його ядро ізоморфне $\bZ_2$.
% Це дає дуальне описання $\pi_0\mathcal{H}(Z,\Delta)$ у термінах простору листів.
% \end{abstract}



\maketitle

\section{Introduction}\label{sect:intro}
In the present paper we study foliations on non-compact surfaces similar to foliations on the plane (considered by W.~Kaplan~\cite{Kaplan:DJM:1940, Kaplan:DJM:1941}, see also~\cite{GodbillonReeb:EM:1966, Godbillon:EM:1972}) with ``sufficiently regular behavior''.
Those foliations are studied in a series of papers by S.~Maksymenko, Ye.~Polulyakh, and Yu.~Soroka, see \cite{MaksymenkoPolulyakh:PGC:2015, MaksymenkoPolulyakh:MFAT:2016, MaksymenkoPolulyakh:PGC:2016, Soroka:MFAT:2016, MaksymenkoPolulyakhSoroka:PICG:2016, Soroka:UMJ:2017, MaksymenkoPolulyakh:PGC:2017}.
Our main result (Theorem~\ref{th:ker_ahom0}) relates the groups of isotopy classes of foliated (\ie sending leaves to leaves) self-homeomorphisms of such surfaces with the groups of isotopy classes of homeomorphisms of the corresponding space of leaves (being in those cases non-Hausdorff one-dimensional manifolds).
This extends results by Yu.~Soroka~\cite{Soroka:UMJ:2017}.

Let $\stripSurf$ be a surface endowed with a one-dimensional foliation $\Partition$.
It will be convenient to say that the pair $(\stripSurf,\Partition)$ is a \myemph{foliated} surface.
For an open subset $\Uman\subset\stripSurf$ denote by $\Partition|_{\Uman}$ the induced foliation on $\Uman$ consisting of connected components of non-empty intersections $\leaf\cap\Uman$ for all $\leaf\in\Partition$.

Let $\Yman = \stripSurf/\Partition$ be the set of all leaves, and $\pr:\stripSurf\to\Yman$ be the natural projection associating to each $x\in\stripSurf$ the leaf of $\Partition$ containing $x$.
Endow $\Yman$ with a quotient topology, so a subset $\Aman\subset\Yman$ is open if and only if $\pr^{-1}(\Aman)$ is open in $\stripSurf$.
Notice that usually spaces of leaves of foliations are non-Hausdorff.

By sthe \myemph{saturation} $\sat{\Uman}$ of a subset $\Uman\subset\stripSurf$ we mean the union of all leaves of $\Partition$ intersecting $\Uman$.
Evidently $\sat{\Uman} = \pr^{-1}(\pr(\Uman))$.
A subset $\Uman\subset\stripSurf$ is called \myemph{saturated} (with respect to a foliation $\Partition$) whenever $\Uman=\sat{\Uman}$.

Let $\Kman$ be a one-dimensional manifold.
Then by a \myemph{trivial} foliation on $\bR\times\Kman$ we will mean a foliation by lines $\{\bR\times y\}_{y\in\Kman}$.

A homeomorphism $h: \stripSurf \to \stripSurf$ is said to be \myemph{leaf-preserving} if $h(\leaf)=\leaf$ for each leaf $\leaf\in\Partition$.
Also, a homeomorphism $h: \stripSurf_1 \to \stripSurf_2$ between foliated surfaces $(\stripSurf_1,\Partition_1)$ and $(\stripSurf_2,\Partition_2)$ is \myemph{foliated} if $h(\leaf)\in\Partition_2$ for each $\leaf\in\Partition_1$.
In particular, every leaf-preserving homeomorphism is foliated.

Let $\HS$ be the group of all foliated self-homeomorphisms of $\stripSurf$ and $\HY$ be the group of homeomorphisms of $\Yman$ endowed with the corresponding compact open topologies.
Let also $\HZS$ be the identity path component of $\HS$ consisting of all homeomorphisms of $\stripSurf$ isotopic to $\id_{\stripSurf}$ in $\HS$, and $\HZY$ be the identity path component of $\HY$.
Evidently, $\HZS$ and $\HZY$ are normal subgroups of the corresponding groups $\HS$ and $\HY$ with respect to composition operation.
Then the quotients
\begin{align*}
	\pi_0\HS &\cong \HS/\HZS, &
	\pi_0\HY &\cong \HY/\HZY,
\end{align*}
\ie the groups of path components of $\HS$ and $\HY$, are called \myemph{homeotopy groups} of the foliation $\Partition$ and $\Yman$ respectively.

Notice that by definition each $\dif\in\HS$ induces a permutation of leaves of $\Partition$, and therefore it induces a bijection $\kdif:\Yman\to\Yman$ making commutative the following diagram:
\begin{equation}\label{equ:psi_diagram}
\xymatrix{
\stripSurf \ar[r]^{\dif} \ar[d]_{\pr} &
\stripSurf \ar[d]^{\pr} \\
\Yman \ar[r]^{\kdif} & \Yman
}
\end{equation}

One easily checks that $\kdif$ is a homeomorphism of $\Yman$ and that the correspondence $\dif\mapsto\kdif$ is a homomorphism of groups
\begin{equation}\label{equ:ahom}
    \ahom: \HS \to \HY, \qquad \ahom(\dif) = \kdif.
\end{equation}

\begin{definition}
We will say that a foliated surface $(\stripSurf,\Partition)$ belongs to the class $\mathcal{F}$, whenever it satisfies the following conditions:
\begin{enumerate}[label={\rm(A\arabic*)}]
\item\label{enum:A1} $\stripSurf$ has a countable base;
\item\label{enum:A2} every leaf $\leaf$ of $\Partition$ is a non-compact closed subset of $\stripSurf$, so it is an image of $\bR$ under some topological embedding $\bR\subset\stripSurf$;
\item\label{enum:A3} each boundary component of $\stripSurf$ is a leaf of $\Partition$.
\end{enumerate}
\end{definition}
Our aim is to show (Theorem~\ref{th:ker_ahom0}) that for foliations from class $\mathcal{F}$ satisfying certain additional ``local finiteness'' condition we have that
\[ \ahom(\HZS) = \HZY. \]
This will imply that the image of $\ahom$ is a union of path components of $\HY$ and that $\ahom$ induces a homomorphism of the corresponding homeotopy groups:
\begin{equation}\label{equ:ahom0}
\ahom_0: \pi_0\HS = \frac{\HS}{\HZS} \longrightarrow \frac{\HY}{\HZY}  = \pi_0\HY.
\end{equation}
Moreover, we also show that in this case for connected $\stripSurf$ the kernel of $\ahom_0$ is either trivial or is isomorphic with $\bZ_2$.

\begin{definition}\label{def:leaves_types}
Let $(\stripSurf,\Partition)$ be a foliated surface from class $\mathcal{F}$.
\begin{enumerate}[label={\rm\arabic*)}, wide, itemsep=1ex]
\item\label{enum:def:leaves_types:regular}
Say that a leaf $\leaf\in\Partition$ is \myemph{regular} if there exists an open saturated neighborhood $\Uman$ of $\leaf$ such that the pair $(\overline{\Uman}, \Uman)$ is ``trivially foliated'' in the sense that $(\overline{\Uman}, \Uman)$ is foliated homeomorphic
\begin{itemize}
\item either to the pair $\bigl( \bR\times[-1,1], \bR\times(-1,1) \bigr)$ if $\leaf\subset\Int{\stripSurf}$,
\item or to the pair $\bigl( \bR\times[0,1], \bR\times[0,1) \bigr)$ if $\leaf\subset\partial\stripSurf$,
\end{itemize}
with trivial foliation into lines $\bR\times t$, via a homeomorphism sending $\leaf$ to $\bR\times 0$.

\item\label{enum:def:leaves_types:singluar}
A leaf which is not regular will be called \myemph{singular}%
\footnote{\,There is some abuse of terminology used in~\cite[Property (c3) before Lemma~3.2]{MaksymenkoPolulyakh:PGC:2015}.
Such leaves for striped surfaces were called \myemph{special}.}.

\item\label{enum:def:leaves_types:special}
For a leaf $\leaf$ define its \myemph{Hausdorff closure} by
\[
   \hcl{\leaf}:= \mathop{\cap}\limits_{\text{$\Uman$ is a neighborhood of $\leaf$}} \overline{\sat{\Uman}},
\]
so it is the intersection of closures of all saturated neighborhoods of $\leaf$.
Evidently, $\hcl{\leaf}$ is saturated and $\leaf\subset \hcl{\leaf}$.
We will say that $\leaf$ is \myemph{special} whenever $\hcl{\leaf} \setminus \leaf \not=\varnothing$.
\end{enumerate}

Denote by $\PartitionReg$, $\PartitionSpec$, and $\PartitionSing$ respectively the families of all regular, special and singular leaves of $\Partition$.
Then we have the following relations:
\begin{align*}
	&\PartitionSpec \subset \PartitionSing, &
	&\Partition = \PartitionReg \sqcup \PartitionSing.
\end{align*}
\end{definition}

\begin{remark}\rm
Notice that by definition each one-dimensional foliation $\Partition$ on a surface $\stripSurf$ is ``locally regular'', \ie every point $x\in\stripSurf$ has a neighborhood $\Uman$ such that the leaf of $\Partition|_{\Uman}$ containing $x$ is regular in the sense of Definition~\ref{def:leaves_types}.
Thus Definition~\ref{def:leaves_types} puts restriction to a global structure of the foliation near a leaf.

Let us also mention that it is possible that a leaf might have a trivially foliated saturated neighborhood $\Uman$, however for any such neighborhood the pair $(\overline{\Uman}, \Uman)$ will never be trivially foliated.
For example, let $\Partition=\{\bR\times y\}_{y\in\bR}$ be a trivial foliation on $\bR^2$ into horizontal lines and $\stripSurf = \bR^2\setminus 0$.
Then it is easy to see that $\Partition|_{\stripSurf}$ has two singular leaves $\omega_1 = (-\infty;0)\times 0$ and $\omega_2 = (0;+\infty)\times 0$.
Indeed, any open saturated neighborhood $\Uman$ of either of $\omega_1$ or $\omega_2$ contains $\bR\times(-\eps,\eps) \setminus (0,0)$ for some small $\eps>0$, and therefore $(\overline{\Uman}, \Uman)$ can not be foliated homeomorphic with $\bigl( \bR\times[-1,1], \bR\times(-1,1) \bigr)$.
This implies $\hcl{\leaf_1} = \hcl{\leaf_2} = \leaf_1 \cup \leaf_2$, whence both $\leaf_1$ and $\leaf_2$ are special.
\end{remark}

The main result of the present paper is the following theorem relating the groups $\pi_0\HS$ and $\pi_0\HY$.
\begin{theorem}\label{th:ker_ahom0}
Let $(\stripSurf,\Partition)$ be a foliated surface from class $\mathcal{F}$ satisfying the following additional condition:
\begin{enumerate}[label={\rm(A\arabic*)}, start=4]
\item\label{enum:A4} the family $\PartitionSing$ of all singular leaves of $\Partition$ is locally finite.
\end{enumerate}
Then $\ahom(\HZS)= \HZY$, whence
\begin{enumerate}[label={\rm(\arabic*)}, leftmargin=*]
\item
the image $\ahom(\HZS)$ of the homomorphism $\ahom:\HS \to \HY$, see~\eqref{equ:ahom}, is a union of path components of $\HY$;
\item
$\ahom$ induces a homomorphism $\ahom_0:\pi_0\HS \to \pi_0\HY$.
\end{enumerate}
Moreover, if $\stripSurf$ is \myemph{connected}, then the kernel $\ker\ahom_0$ is either
\begin{enumerate}[label={\rm(\alph*)}]
\item \myemph{trivial}, \ie $\ahom_0$ is injective, or % and we can identify $\pi_0\HS$ with a subgroup of $\pi_0\HY$, or
\item $\ker\ahom_0 \cong \bZ_2$, which happens if and only if $\HS$ contains a leaf-preserving homeomorphism which reverses orientation of all leaves.
\end{enumerate}
\end{theorem}

In fact in a series of papers by the authors it was obtained a characterization of foliated surfaces $(\stripSurf,\Partition)$ from class $\mathcal{F}$ satisfying also the assumption~\ref{enum:A4} of Theorem~\ref{th:ker_ahom0}.
It is shown in \cite[Theorem~4.4]{MaksymenkoPolulyakh:PIGC:2017} that every such surface is glued from certain ``model strips'' foliated by parallel lines.
Such surfaces were called \myemph{striped}, and the proof of Theorem~\ref{th:ker_ahom0} essentially exploits an existence of that gluing.
This description is motivated by results of W.~Kaplan, see Theorem~\ref{th:Kaplan} below, describing the structure of foliations on $\bR^2$.

Furthermore, in the joint paper with Yu.~Soroka~\cite{MaksymenkoPolulyakhSoroka:PICG:2016} the homeotopy group $\pi_0\HS$ of canonical foliation $\Partition$ of a striped surface $\stripSurf$ is identified with the group of automorphisms of a certain graph $G$ with additional structure encoding the gluing of $\stripSurf$ from model strips.
In fact, the vertices of $G$ are strips and the edges are ``seams'' (\ie leaves along which we glue the strips), so this graph is \myemph{in a certain sense dual} to the space of leaves $\Yman$.
It is also proved in S.~Maksymenko and O.~Nikitchenko~\cite{MaksymenkoNikitchenko:2021} that $G$ (as a one-dimensional CW-complex) is homotopy equivalent to the surface $\stripSurf$.

Thus Theorem~\ref{th:ker_ahom0} also gives a dual description of the homeotopy group $\pi_0\HS$ in terms of the space of leaves.

\subsection*{Structure of the paper}
In Section~\ref{sect:nonHausdorff_1-manif} we consider several properties of \myemph{non-Hausdorff} $T_1$-spaces.
In particular, we recall the exponential law for them as well as describe the relations between isotopies and paths in the groups of homeomorphisms for such spaces, see Corollary~\ref{cor:CAB_homotopies}.
We also present a characterization of the identity path component of the group of homeomorphisms of second countable $T_1$-manifolds that are not necessarily Hausdorff, see Lemma~\ref{lm:spec_pt_manif}.
In Section~\ref{sect:striped_surf} we recall the results about striped surfaces, and prove Theorem~\ref{th:ker_ahom0} in Section~\ref{proof:th:ker_ahom0}.

\section{$T_1$-spaces}\label{sect:nonHausdorff_1-manif}
We will recall here several elementary properties of $T_1$-spaces which are not necessarily Hausdorff.
They are well-known and rather trivial for $T_2$-spaces, while for $T_1$-ones they should also be known for specialists, thought we did not find their explicit exposition in the literature.
Therefore to make the paper self-contained and for future references we will collect them together and present their short proofs.
These resutls will be applied to the spaces of leaves of foliations.

Recall that a topological space $\Yman$ is \myemph{compact} if every open cover of $\Yman$ contains a finite subcover.
In this paper a (not necessarily Hausdorff) topological space $\Yman$ will be called \myemph{locally compact}%
\footnote{\,Notice that there are several variations of this notion for non-Hausdorff spaces which coincide for Hausdorff ones.}
whenever for each $y\in \Yman$ and an open neighborhood $\Uman$ of $y$ there exists a compact subset $B\subset\Yman$ such that $y\in \Int{B} \subset \overline{B} \subset \Uman$.

\subsection{Exponential law}
For two sets $\Aman$ and $\Bman$ denote by $\maps{\Aman}{\Bman}$ the set of all maps $f:\Aman\to\Bman$.
Then for any other set $\Cman$ we have a natural bijection
\[
E = \EXP{\Aman,\Bman,\Cman}: \maps{\Aman\times\Bman}{\Cman} \to  \maps{\Aman}{\maps{\Bman}{\Cman}}
\]
called an \myemph{exponential map} and defined as follows: if $\Fmap:\Aman\times\Bman\to\Cman$ is a map, then $E(\Fmap):\Aman \to \maps{\Bman}{\Cman}$ is given by $E(\Fmap)(a)(b) = \Fman(a,b)$ for all $(a,b)\in\Aman\times\Bman$.

If $\Aman,\Bman$ are topological spaces, it is natural to consider the space $\Cont{\Aman}{\Bman}$ of continuous maps $\Aman\to\Bman$ instead of $\maps{\Aman}{\Bman}$.
We will always endow $\Cont{\Aman}{\Bman}$ with the compact open topology.
Then the variant of exponential law for topological spaces holds under additional assumptions of $\Aman$ and $\Bman$:
\begin{subtheorem}\label{th:exp_law}
Let $\Aman$, $\Bman$, $\Cman$ be any topological spaces.
Then $E$ induces a continuous injective map
\begin{equation}\label{equ:exp_law}
E = \EXP{\Aman,\Bman,\Cman}: \Cont{\Aman\times\Bman}{\Cman} \to \Cont{\Aman}{\Cont{\Bman}{\Cman}}.
\end{equation}
If $\Bman$ is locally compact (not necessarily Hausdorff), then~\eqref{equ:exp_law} is a continuous bijection.
Moreover, if, in addition, $\Aman$ is a $T_3$-space, then~\eqref{equ:exp_law} is a homeomorphism.
\end{subtheorem}
This theorem is well known and usually formulated for Hausdorff spaces~\cite{Engelking:GenTop}.
For non-Hausdorff spaces its proof is presented in~\cite[Lemma~3.3]{KhokhliukMaksymenko:PIGC:2020}.

\begin{subcorollary}\label{cor:CAB_homotopies}
If $\Bman$ is locally compact, and $\Cman$ is an arbitrary topological space, then the exponential map
\begin{equation}\label{equ:exp_law_paths}
E = \EXP{[0,1],\Bman,\Cman}: \Cont{[0,1]\times\Bman}{\Cman} \to \Cont{[0,1]}{\Cont{\Bman}{\Cman}}
\end{equation}
is a homeomorphism.
In particular, it induces a bijection
\[
E_0: \pi_0\Cont{[0,1]\times\Bman}{\Cman} \to \pi_0\Cont{[0,1]}{\Cont{\Bman}{\Cman}}.
\]
between the corresponding sets of path components.
\end{subcorollary}

Moreover, for groups of homeomorphisms one can say more.
Let
\begin{itemize}[leftmargin=*]
\item $\Homeo(\Bman)$ be the group of homeomorphisms of a topological space $\Bman$;
\item $\HomeoId(\Bman)$ be the path component of $\id_{\Bman}$ in $\Homeo(\Bman)$, so it consists of $\dif\in\Homeo(\Bman)$ such that there exists a continuous path $\gamma:[0,1]\to\Homeo(\Bman)$ such that $\gamma(0)=\id_{\Bman}$ and $\gamma(1)=\dif$;
\item $\tHomeoId{\Bman}$ be the subgroup of $\Homeo(\Bman)$ consisting of homeomorphisms isotopic to $\id_{\Bman}$.
\end{itemize}
\begin{subcorollary}\label{cor:pi0_homeo}
For any topological space $\Bman$ we have that
\begin{enumerate}[label={\rm(\arabic*)}, leftmargin=*]
\item\label{enum:cor:pi0_homeo:Hid}
$\tHomeoId{\Bman} \subset \HomeoId(\Bman)$ and these subgroups are normal in $\Homeo(\Bman)$;
\item\label{enum:cor:pi0_homeo:loc_comp}
if $\Bman$ is locally compact, $\tHomeoId{\Bman} = \HomeoId(\Bman)$;

\item\label{enum:cor:pi0_homeo:pi0}
the set $\pi_0\Homeo(\Bman)$ of path components of $\Homeo(\Bman)$ has a group structure under which the natural map $p:\Homeo(\Bman) \to \pi_0\Homeo(\Bman)$ is an epimorphism with kernel $\HomeoId(\Bman)$, whence $p$ induces an isomorphism of groups $\Homeo(\Bman)/\HomeoId(\Bman) \cong \pi_0\Homeo(\Bman)$;
\end{enumerate}
\end{subcorollary}
\begin{proof}
\ref{enum:cor:pi0_homeo:Hid}
By Theorem~\ref{th:exp_law} every isotopy $\Fman:[0,1]\times\Bman\to\Bman$ induces a continuous path $E(\Fman):[0,1]\to\Homeo(\Bman)$ between $\Fman_0$ and $\Fman_1$.
This implies that $\tHomeoId{\Bman} \subset \HomeoId(\Bman)$.
Normality of these groups is easy.

\ref{enum:cor:pi0_homeo:loc_comp}
If $\Bman$ is locally compact, then by Corollary~\ref{cor:CAB_homotopies} the exponential map $\EXP{[0,1],\Bman,\Bman}$ is a homeomorphism, so every continuous path $\gamma:[0,1]\to\Homeo(\Bman)$ is induced by some isotopy $E^{-1}(\gamma):[0,1]\times\Bman\to\Bman$.
This gives the inverse inclusion $\tHomeoId{\Bman} \supset \HomeoId(\Bman)$.

Statement~\ref{enum:cor:pi0_homeo:pi0} is straightforward.
\end{proof}

Thus for locally compact spaces isotopies can equally be thought of as paths in the group of homeomorphisms endowed with compact open topologies.

\subsection{Special points of $T_1$ spaces}
Let $\Yman$ be a topological space.

\begin{subdefinition}\label{def:hcl}
For a point $z\in\Yman$ define its \myemph{Hausdorff closure}, $\hcl{z}$, to be the intersection of closures of all neighborhoods of $z$, that is
\begin{equation}\label{equ:hcl}
	\hcl{y} = \bigcap_{\text{$\Vman$ is a neighborhood of $y$}} \overline{\Vman}.
\end{equation}
Evidently, $z\in\hcl{z}$.
We say that $z$ is \myemph{\branch}%
\footnote{In \cite{HaefligerReeb:EM:1957, GodbillonReeb:EM:1966} such points were called \myemph{branch}.}
whenever $\hcl{z} \setminus z \not=\varnothing$.
Denote by $\specPoints$ the set of all \branch\ points of $\Yman$.
\end{subdefinition}

One easily check that $\Yman$ is Hausdorff iff $\hcl{z} = \{z\}$ for all $z\in\Yman$, that is $\specPoints=\varnothing$.

Further, let $\Lman\subset\Yman$ be a subset, $y\in \Lman$, and $\hclA{y}{\Lman}$ be Hausdorff closure of $z$ in $\Lman$ with respect to the induced topology from $\Yman$.
It is straightforward that
\begin{equation}\label{equ:hclA}
\hclA{y}{\Lman} \ \subset \ \hcl{y} \cap \Lman,
\end{equation}
however the opposite inclusion can fail.

\begin{sublemma}\label{lm:specpt}
For a subset $\Aman\subset\Yman$ the following statements hold true.
\begin{enumerate}[label={\rm(\arabic*)}, itemsep=0.5ex, leftmargin=*]
\item\label{enum:lm:specpt:Y_is_h}
If $\Aman \cap \hcl{y} = \{y\}$ for all $y\in \Aman$, then $\Aman$ is Hausdorff.

\item\label{enum:lm:specpt:A_is_h}
In particular, for every $z\in \specPoints$ the following subspace
\[
\Aman_z := (\Aman\setminus\specPoints)\cup \{z\} =
\Aman \setminus \bigl( \specPoints \setminus \{z\} \bigr)
\]
of $\Yman$ is Hausdorff.

\item\label{enum:lm:specpt:comp}
If $\Aman$ is compact, then $\overline{\Aman} \subset  \Aman \cup \specPoints$.
\end{enumerate}
\end{sublemma}
\begin{proof}
Statement~\ref{enum:lm:specpt:Y_is_h} is an immediate consequence of~\eqref{equ:hclA}.

\ref{enum:lm:specpt:A_is_h}
To prove that $B := (\Lman\setminus\specPoints)\cup \{z\}$ is Hausdorff it suffices to verify that $\hclA{y}{B}=\{y\}$ for all $y\in B$.

If $y\in \Lman \setminus \specPoints$, then $\hcl{y} = \{y\}$, whence $B \cap \hcl{y} = \{y\}$.
Since $z \in \hcl{w}$ if and only if $w \in \hcl{z}$, it follows that $\hcl{z} \subset \specPoints$.
Therefore
\begin{align*}
	B \cap \hcl{z} &=
	\bigl((\Lman\setminus\specPoints)\cup \{z\} \bigr) \cap \hcl{z} =\\
	&=
	\bigl( (\Lman\setminus\specPoints) \cap \hcl{z} \bigr)
	\cup
	\bigl(  \{z\} \cap \hcl{z}\bigr)=
	\varnothing \cup \{z\} = \{z\}.
\end{align*}

\ref{enum:lm:specpt:comp}
Let $y \notin \Lman   \cup \specPoints$.
We will show that then there exists an open neighborhood $W$ of $y$ such that $\Lman \cap W = \varnothing$.
This will imply that $y\not\in\Cl{\Lman}$, whence $\Cl{\Lman} \subset \Lman \cup \specPoints$.

Since $y$ is not a \branch\ point, \ie $\hcl{y} = \{y\}$, we get from~\eqref{equ:hcl} that
\[
\Lman \ \subset \ \Lman \cup \specPoints \ \subset \ \Yman \setminus \{y\} \ = \ \bigcup_{V \ni y} (\Yman \setminus \Cl{V}),
\]
where $V$ runs over all open neighborhoods of $y$.
Thus
\[
\{
\Yman \setminus \overline{V} \mid V \ \text{is a neighborhood of $y$}
\}
\]
is an open cover of a compact set $\Lman$, and so it contains a finite subcover, \ie one can find open neighborhoods $V_1,\ldots,V_m$ of $y$ such that $\Lman \subset \bigcup_{i = 1}^m (\Yman \setminus \Cl{V_i})$.
Hence $W = V_1 \cap \ldots \cap V_m$ is an open neighborhood of $y$ with $\Lman \cap W = \emptyset$.
\end{proof}

Evidently, statement~\ref{enum:lm:specpt:comp} of Lemma~\ref{lm:specpt} extends the well known fact that a compact subset $\Aman$ of a Hausdorff space $\Yman$ is closed.

\begin{subcorollary}\label{cor.compact_is_closed}
	Let $\Aman$ be a compact subset of the subspace $\Yman \setminus \specPoints$.
	Then it is closed in $\Yman$.
\end{subcorollary}

\begin{proof}
Let $x \in \overline{\Aman} \setminus \Aman \subset \specPoints$.
Consider the Hausdorff subspace $\Aman_x$ of $Y$.
Then $\Aman$ is closed in $\Aman_x$ as a compact subset of the Hausdorff space $\Aman_x$, whence $\Aman = \overline{\Aman} \cap \Aman_x$.
On the other hand $x \in \overline{\Aman} \cap \Aman_x$ by definition which is impossible.
\end{proof}

\subsection{Locally finite subsets}\label{sect:loc_fin_subsets}
Say that a subset $\Aman \subset \Yman$ is \myemph{locally finite} if for every point $z\in\Yman$ there exists an open neighborhood $\Uman$ such that the intersection $\Uman\cap \Aman$ is finite.
In other words, the family $\bigl\{ \{a\} \mid a\in \Aman \bigr\}$ of one-point subsets of $\Aman$ is locally finite in $\Yman$.

\begin{sublemma}\label{lm:T1_locfin}
Consider the following conditions on a subset $\Aman$ of a topological space $\Yman$:
\begin{enumerate}[label={\rm(\arabic*)}, itemsep=0.5ex]
\item\label{enum:lm:T1_locfin:closed_discr}
$\Aman$ is closed and discrete;

\item\label{enum:lm:T1_locfin:U}
for each $y\in\Yman$ there exists a neighborhood $\Vman$ intersecting $\Aman$ in at most one point;

\item\label{enum:lm:T1_locfin:SP_discrete}
$\Aman$ is locally finite;
\end{enumerate}
Then we have the following implications:
\ref{enum:lm:T1_locfin:closed_discr}%
$\Rightarrow$%
\ref{enum:lm:T1_locfin:U}%
$\Rightarrow$%
\ref{enum:lm:T1_locfin:SP_discrete}.

If $\Yman$ is $T_1$ then we also have that~\ref{enum:lm:T1_locfin:SP_discrete}$\Rightarrow$\ref{enum:lm:T1_locfin:closed_discr}, \ie all the above conditions are equivalent.
\end{sublemma}
\begin{proof}
\ref{enum:lm:T1_locfin:closed_discr}$\Rightarrow$\ref{enum:lm:T1_locfin:U}.
Suppose $\Aman$ is closed and discrete and let $y\in\Yman$.
If $y\not\in\Aman$, then $\Vman = \Yman\setminus\Aman$ is an open neighborhood of $y$ that does not intersect $\Aman$.
If $y\in\Aman$, then discreteness of $\Aman$ implies that there exists an open neighborhood $\Vman$ of $y$ such that $\Vman\cap\Aman = \{y\}$.

The implication \ref{enum:lm:T1_locfin:U}$\Rightarrow$\ref{enum:lm:T1_locfin:SP_discrete} is evident.

It remains to prove the implication \ref{enum:lm:T1_locfin:SP_discrete}$\Rightarrow$\ref{enum:lm:T1_locfin:closed_discr} under the assumption that $\Yman$ is $T_1$.
Suppose $\Aman$ is locally finite.
Then each subset $B\subset\Aman$ is locally finite as well.
Moreover, as every point $y\in\Yman$ is a closed subset, it follows that $B$ is closed as a union of a locally finite family of its closed one-point subsets.
In other words every subset of $\Aman$ is closed in $\Yman$.
Hence $\Aman$ is closed and discrete.
\end{proof}

\subsection{One-dimensional $T_1$ manifolds}
Let $\Yman$ be a $T_1$ topological space locally homeomorphic with $[0,1)$.
In other words, $\Yman$ is a one-dimensional manifold which is $T_1$ but not necessarily Hausdorff.
Then $\Yman$ is also locally compact, and by Corollary~\ref{cor:pi0_homeo} one can equally regard $\pi_0\HY$ as the group of path components of $\HY$ and also as the group of isotopy classes of homeomorphisms of $\Yman$.

As usual, a point $y\in\Yman$ is called \myemph{internal}, if it has an open neighborhood homeomorphic with $(0,1)$.
Otherwise, $y$ has an open neighborhood homeomorphic with $[0,1)$ and is called a \myemph{boundary} point.
As usual, the sets of all internal and boundary points will be denoted by $\Int{\Yman}$ and $\partial\Yman$ respectively.

The following Lemma~\ref{lm:spec_pt_manif} characterizes the 	identity path component of $\Yman$ under assumption that the set of its \branch\ points is locally finite.

\begin{sublemma}\label{lm:spec_pt_manif}
Let $\Yman$ be a $T_1$ space locally homeomorphic with $[0,1)$ such that the set $\specPoints$ of its \branch\ points is locally finite and every connected component of the complement $A := \Yman\setminus\specPoints$ is second-countable.
Let also $z\in\specPoints$ be any \branch\ point and $A_z:=A\cup\{z\}$.
Then the following statements hold.
\begin{enumerate}[label={\rm(\arabic*)}, leftmargin=*, itemsep=1ex]
\item\label{enum:lm:spec_pt_manif:A_Az_manif}
$A$ and $A_z$ are open and Hausdorff.

\item\label{enum:lm:spec_pt_manif:A_ha_open_comp}
Every connected component $\dComp$ of $A$ or of $A_z$ is open in $\Yman$ and is homeomorphic to one of the following spaces: the circle $S^1$, $[0,1]$, $[0,1)$, $(0,1)$.
In particular, $\dComp$ is orientable.

\item\label{enum:lm:spec_pt_manif:compactness}
If a connected component $\dComp$ of $A$ is compact then it is a connected component of $\Yman$.

\item\label{enum:lm:spec_pt_manif:isotopies}
Let $\kdif\in\HY$.
Then $\kdif\in\HZY$ if and only if
\begin{enumerate}[leftmargin=*, label={\rm(Y\Alph*)}]
\item\label{enum:cond_Hid:z}
$\kdif$ is fixed on $\Tman:=\partial\Yman\cup\specPoints$;

\item\label{enum:cond_Hid:C}
$\kdif(\dComp)=\dComp$ for every connected component $\dComp$ of $\Yman\setminus\Tman$, and the restriction map $\kdif|_{\dComp}:\dComp\to \dComp$ is isotopic to $\id_{\dComp}$ (which in the case of $1$-dimensional manifolds is equivalent to the assumption that $\kdif|_{\dComp}$ preserves orientation of $\dComp$).
\end{enumerate}
If these conditions hold, then $\kdif$ is also fixed on $\partial\Yman$.
\end{enumerate}
\end{sublemma}
\begin{proof}
\ref{enum:lm:spec_pt_manif:A_Az_manif}
By Lemma~\ref{lm:specpt}, $A$ and $A_z$ are \myemph{Hausdorff}.
Moreover, as $\Yman$ is $T_1$ and $\specPoints$ is locally finite, it follows that $\specPoints$ and $\specPoints\setminus\{z\}$ are closed in $\Yman$, whence their complements $A$ and $A_z$ are \myemph{open}.

\ref{enum:lm:spec_pt_manif:A_ha_open_comp}
This statement is proved in~\cite[Lemma~2.3]{MaksymenkoPolulyakh:MFAT:2016}.
Let us recall the arguments.
Since $A_z$ is open in $\Yman$ and $\Yman$ is locally path connected, it follows that every connected component $\dComp$ of $A$ and $A_z$ is open in $\Yman$ as well.
Moreover, by assumption every connected component of $A$ is second countable.
As $z$ has a countable base of open neighborhoods each being a union of $z$ and at most two open subsets of $A$, it follows that the connected components of $A_z$ are also \myemph{second countable}.
Hence $\dComp$ is a connected Hausdorff $1$-manifold, and the results follows from the classification of such manifolds, \eg \cite[Exercises~1.2.6 and~1.4.9]{Hirsch:DiffTop}.

\ref{enum:lm:spec_pt_manif:compactness}
The connected set $\dComp$ is both open, see~\ref{enum:lm:spec_pt_manif:A_Az_manif} and closed, see Corollary~\ref{cor.compact_is_closed}.

\ref{enum:lm:spec_pt_manif:isotopies}
\myemph{Necessity.}
Let $\kdif\in\HZY$, so there is an isotopy $\Fmap:[0,1]\times\Yman\to\Yman$ such that $\Frestr{0}=\id_{\Yman}$ and $\Frestr{1}=\kdif$.
We need to check conditions~\ref{enum:cond_Hid:z} and~\ref{enum:cond_Hid:C} for $\kdif$.

First notice that $\specPoints$ and $\partial\Yman$ are defined in topological terms, \ie they are invariant under self-homeomorphisms of $\Yman$.
Moreover, each $z\in\partial\Yman\setminus\specPoints$ has a neighborhood $\Uman_z$ such that $\Uman_z\cap\Tman = \{z\}$.
Since $\specPoints$ is locally finite, and therefore discrete due to Lemma~\ref{lm:T1_locfin}, we obtain that $\Tman$ is discrete as well.
This implies that $\Fmap([0,1]\times z) \subset \Tman$ for any $z\in\Tman$, and therefore $\Fmap([0,1]\times z)$ is contained in some connected component of $\Tman$, being a one-point set since $\Tman$ is discrete.
Hence $\Frestr{t}(z) =\Frestr{0}(z) = z$ for all $t\in[0,1]$.
In particular, $\kdif$ is fixed on $\Tman$, which proves~\ref{enum:cond_Hid:z}.

As $\Yman\setminus\Tman$ is also invariant under self-homeomorphisms of $\Yman$, we obtain by similar arguments that $\Fman([0,1]\times \dComp) \subset\dComp$ for every connected component $\dComp$ of $\Yman\setminus\Tman$.
Thus the restriction
\[ \Fman|_{[0,1]\times \dComp}:[0,1]\times \dComp\to\dComp\]
is an isotopy between $\id_{\dComp}$ and $\kdif|_{\dComp}$.
Therefore, $\kdif|_{\dComp}:\dComp\to\dComp$ preserves orientation of $\dComp$ which proves~\ref{enum:cond_Hid:C}.

\myemph{Sufficiency.}
Let $\kdif\in\HY$ be a homeomorphism satisfying~\ref{enum:cond_Hid:z} and~\ref{enum:cond_Hid:C}.
For every connected component $\dComp$ of $\Aman=\Yman\setminus\Tman$ fix an arbitrary isotopy $\Fmap^{\dComp}:[0,1]\times\dComp\to\dComp$ such that $\Frestr{0}^{\dComp}=\id_{\dComp}$ and $\Frestr{1}^{\dComp}=\kdif|_{\Cman}$ and define the map $\Gmap:[0,1]\times\Yman\to\Yman$ by
\[
\Gmap(t,\ptx) =
\begin{cases}
\ptx, & \ptx\in\Tman, \\
\Fmap^{\dComp}(t,\ptx), & \text{$\ptx\in\dComp$, where $\dComp$ is a connected component of $\Yman\setminus\Tman$}.
\end{cases}
\]

Then $\Grestr{0} =\id_{\Yman}$, $\Grestr{1} =\kdif$, and each $\Grestr{t}$ is a bijection of $\Yman$.
Notice that $\Gman|_{\Aman\times[0,1]}$ is an isotopy of $\Aman$, and $\Gman$ is its extension to $\Yman\times[0,1]$ fixed at \branch\ points.
We need to show that \myemph{$\Gman$ is an isotopy of $\Yman$}.

For each $z\in\Tman$ let $\dComp_{z}$ be a connected component of $\Aman_z := \Aman\cup\{z\}$ containing $z$.
Then $\Aman \cup \{\dComp_z\}_{z\in\Tman}$ is an open cover of $\Yman$.
Since the restriction $\Gman|_{\Aman\times[0,1]}:\Aman\times[0,1]\to\Aman$ is an isotopy, it suffices to show that for each $z\in\Tman$, the restriction $\Gman:\dComp_{z}\times[0,1]\to\dComp_{z}$ is an isotopy as well.

If $\Cman_z$ is compact, \ie it is homeomorphic either to $S^1$ or to $[0,1]$, then $\Cman_z$ is a one-point compactification of $\Cman$, and continuity of $\Gman|_{\dComp_{z}\times[0,1]}$ follows from~\cite[Theorem~2.1]{Crowell:ProcAMS:1963}.

Otherwise, $(\Cman_z,z)$ is homeomorphic to one of the following pointed spaces:
\[ \bigl((-1,1), 0\bigr),  \qquad \bigl([-1,1), 0\bigr), \qquad  \bigl([0,1), 0\bigr). \]
In this case the proof can proceed similarly to the arguments of~\cite[Theorem~3.2]{Crowell:ProcAMS:1963} extending the previous case one-point compactification of $\Cman$ to the case of the so-called ideal Freudental compactification.
In fact, the proof of~\cite[Theorem~3.2]{Crowell:ProcAMS:1963} consists of proving continuity of $\Gmap$ for each added point, which in our situation is $z$ only.
\end{proof}

\begin{subremark}
The problem of one-point extensions of isotopies (as in the proof of continuity of $\Gmap$) were studied by R.~H.~Crowell~\cite[Theorem~2.1]{Crowell:ProcAMS:1963}, H.~Gluck~\cite{Gluck:BAMS:1963} and many others.
Also, certain variants of such results for concordancies of homeomorphisms considered in L.~S.~Husch and T.~B.~Rushing~\cite{HuschRushing:MichJM:1969}, applications to ambient isotopies of arcs in manifolds~\cite{Feustel:ProcAMS:1966, MartinRolfsen:ProcAMS:1968}, and applications to Nielsen numbers are given in P.~Heath and X.~Zhao~\cite[\S4]{HeathZhao:TopApp:1997}.
\end{subremark}

\section{Striped surfaces}\label{sect:striped_surf}
In this section we will briefly recall the results by W.~Kaplan, S.~Maksymenko, E.~Polulyakh, and Yu.~Soroka.

\subsection{Quasi-striped surfaces}
The following notion of a quasi-strip is a ``basic block'' generating foliations on the plane.
\begin{subdefinition}\label{def:quasi_strip}
A foliated surface $(\strip,\Partition_{\strip})$ will be called a \myemph{quasi-strip} whenever
\begin{enumerate}[label={\rm\alph*)}, leftmargin=*, itemsep=1ex]
\item\label{enum:quasi_strip:1}
every boundary component of $\strip$ is an open arc being also a leaf of $\Partition_{\strip}$;
\item\label{enum:quasi_strip:2}
the restriction of $\Partition_{\strip}$ to the interior $\Int{\strip}$ is foliated homeomorphic with $\bR^2$ endowed by trivial foliation into horizontal lines $\{\bR\times y\}_{y\in\bR}$.
\end{enumerate}
\end{subdefinition}

\begin{subdefinition}
Let $\Partition=\{\bR\times y\}_{y\in[0,1]}$ be the trivial foliation on $\bR\times[0,1]$ into horizontal lines.
Then for every \myemph{open} (in the induced topology) subset $\strip\subset\bR\times[0,1]$ such that $\bR\times(0,1) \subset \strip$ the foliated surface $(\strip,\Partition|_{\strip})$ will be called a \myemph{model strip}.
In this case we will use the following notation:
\begin{align*}
	\bSide{0}{\strip} &:= \strip \, \cap \, (\bR \times \lbrace 0 \rbrace), &
    \bSide{1}{\strip} &:= \strip \, \cap \, (\bR \times \lbrace 1 \rbrace), \\
    \partial\strip    &:= \bSide{0}{\strip} \, \cup \, \bSide{1}{\strip}, &
    \Int{\strip}      &:= \bR\times(0,1),
\end{align*}
and call $\partial\strip$ the \myemph{boundary} of $\strip$, and $\bSide{0}{\strip}$ and $\bSide{1}{\strip}$ the \myemph{boundary sides} of $\strip$.
\end{subdefinition}
Thus a model strip is a union of $\bR\times(0,1)$ with at most countably many (possibly zero) open mutually disjoint intervals contained in $\bR\times\{0,1\}$.
Evidently, each model strip is a quasi-strip as well.
The difference between these two notions is that in the model strip all boundary components are distributed between two boundary sides only, while quasi-strips may have boundary components between internal leaves.

\begin{subdefinition}\label{def:quasi-striped_atlas}
Let $\stripSurf$ be a two-dimensional topological manifold (a surface) and $\preStripSurf = \bigsqcup \limits_{\stInd \in \StInd} \strip_{\stInd}$ be a family of mutually disjoint \myemph{quasi-strips} $\{\strip_{\stInd}\}_{\stInd \in \StInd}$.
A \myemph{quasi-striped atlas} on $\stripSurf$ is a map $\qmap: \preStripSurf \to \stripSurf$ such that
\begin{enumerate}[leftmargin=*, label={\rm\arabic*)}, itemsep=1ex]
\item\label{enum:def:quasi-striped_atlas:1}
$\qmap$ is a \myemph{quotient} map, \ie it is continuous, surjective, and a subset $\Uman\subset\stripSurf$ is open if and only if $\qmap^{-1}(\Uman) \cap \strip_{\stInd}$ is open in $\strip_{\stInd}$ for each $\stInd\in\StInd$;

\item\label{enum:def:quasi-striped_atlas:2}
there exist two disjoint families $\mathcal{X} = \{\bdX_\bdGlueInd\}_{\bdGlueInd\in\BdGlueInd}$ and $\mathcal{Y} = \{\bdY_\bdGlueInd\}_{\bdGlueInd\in\BdGlueInd}$ of mutually distinct boundary intervals of $\preStripSurf$ enumerated by the same set of indexes $\BdGlueInd$ such that
\begin{enumerate}[leftmargin=*, label={\rm(\alph*)}, itemsep=1ex]
\item\label{enum:def:quasi-striped_atlas:2:a}
$\qmap$ is injective on $\preStripSurf \setminus (\mathcal{X} \cup \mathcal{Y})$;
\item\label{enum:def:quasi-striped_atlas:2:b}
$\qmap(\bdX_\bdGlueInd) = \qmap(\bdY_\bdGlueInd)$ for each $\bdGlueInd\in\BdGlueInd$, and the restrictions
\[
    \bdX_\bdGlueInd \xrightarrow{~\qmap|_{\bdX_\bdGlueInd}~} \qmap(\bdX_\bdGlueInd) = \qmap(\bdY_\bdGlueInd)
    \xleftarrow{~\qmap|_{\bdY_\bdGlueInd}~}  \bdY_\bdGlueInd
\]
are bijections.
\end{enumerate}
\end{enumerate}
The images $\qmap(\bdX_\bdGlueInd) = \qmap(\bdY_\bdGlueInd)$, $\bdGlueInd\in\BdGlueInd$, will be called \myemph{seams} of $\stripSurf$ (as well as of the striped atlas $q$).

If all strips $\{\strip_{\stInd}\}_{\stInd \in \StInd}$ are model, then the quasi-striped atlas $q$ will be called \myemph{striped}.
\end{subdefinition}
Condition~\ref{enum:def:quasi-striped_atlas:2:b} implies that for each $\bdGlueInd \in \BdGlueInd$ we have a well-defined ``gluing'' homeomorphism
$\phi_{\bdGlueInd} = \bigl(\qmap|_{\bdX_{\bdGlueInd}}\bigr)^{-1} \circ  \qmap|_{\bdY_{\bdGlueInd}}: \ \bdY_{\bdGlueInd} \to \bdX_{\bdGlueInd}$.
Thus a quasi-striped surface is obtained from a family of quasi-strips by gluing them along certain pairs of boundary intervals by homeomorphisms $\phi_{\bdGlueInd}$.
It is allowed to glue two quasi-strips along more than one pair of boundary components, and one may also glue boundary components belonging to the same quasi-strip.

Notice that the foliations $\Partition_{\strip_{\stInd}}$ on strips $\strip_{\stInd}$ give a unique foliation $\Partition_0$ on $\preStripSurf$.
In this case if $\gamma$ is a leaf of $\Partition_0$ contained in a strip $\strip_{\stInd}$, then we will say that \myemph{$\strip_{\stInd}$ is a strip of $\gamma$}.

Also, the images of leaves of $\Partition_0$ under $\qmap$ constitute a one-dimensional foliation $\Partition$ on $\stripSurf$ which we will call the \myemph{canonical} foliation associated with the quasi-striped atlas $\qmap$.

\begin{subdefinition}
A foliated surface $(\stripSurf,\Partition)$ will be called a \myemph{(quasi-)striped} surface whenever it admits a (quasi-)striped atlas whose canonical foliation is $\Partition$.
\end{subdefinition}

Now the results by W.~Kaplan can be formulated as follows.
\begin{subtheorem}[W.~Kaplan~\cite{Kaplan:DJM:1940, Kaplan:DJM:1941}]\label{th:Kaplan}
Every one-dimensional foliation $\Partition$ on the plane $\bR^2$ belongs to the class $\mathcal{F}$ and admits a quasi-striped atlas, \ie $\Partition$ is a canonical foliation associated with some quasi-striped atlas $\qmap$ on $\bR^2$.
\end{subtheorem}

\begin{subremark}\rm
Let $(\stripSurf,\Partition)$ be any (oriented or non-oriented, compact or non-compact) connected surface distinct from $S^2$ and $\bR{P}^2$.
Then the universal covering of its interior $\Int{\stripSurf}$ is homeomorphic with $\bR^2$, \eg \cite[Corollary~1.8]{Epstein:AM:1966}.
Moreover, if $p:\bR^2 \to \Int{\stripSurf}$ is the corresponding universal covering map, then one has a foliation $F$ on $\bR^2$ induced from $\Partition$: the leaves of $F$ are connected components of the inverse images $p^{-1}(\leaf)$, $\leaf\in\Partition$.
There is also a natural free action of the fundamental group $\pi_1\stripSurf$ on $\bR^2$ by foliated homeomorphisms so that $(\Int{\stripSurf},\Partition|_{\Int{\stripSurf}})$ is the quotient $(\bR^2,F)/\pi_1\stripSurf$.

This observation shows that foliations on $\bR^2$ determine all foliations on all connected boundaryless surfaces distinct from $S^2$ and $\bR{P}^2$.
However, this does not help so much, since usually it is hard to explicitly describe the action of $\pi_1\stripSurf$ of $\bR^2$ by covering transformations.

Also, due to Theorem~\ref{th:Kaplan} the foliated surface $(\bR^2, F)$ admits a quasi-striped atlas.
However, again such a splitting into quasi-strips in general is not related with the action of $\pi_1\stripSurf$.
Therefore it is still better to work with the surface $\stripSurf$ instead of $\bR^2$.
\end{subremark}

In fact, the splitting of a quasi-striped surface into quasi-strips proposed by Kaplan in Theorem~\ref{th:Kaplan} is not unique.
This led the present authors to consider quasi-striped surfaces glued from model strips only, and it turned out that splitting into such strips (if it exists) is unique, see Theorem~\ref{th:reduced_striped_surf}.

\begin{subremark}
Let us discuss distinct types of leaves of $\Partition$.
\begin{figure}[htbp!]
\includegraphics[width=\textwidth]{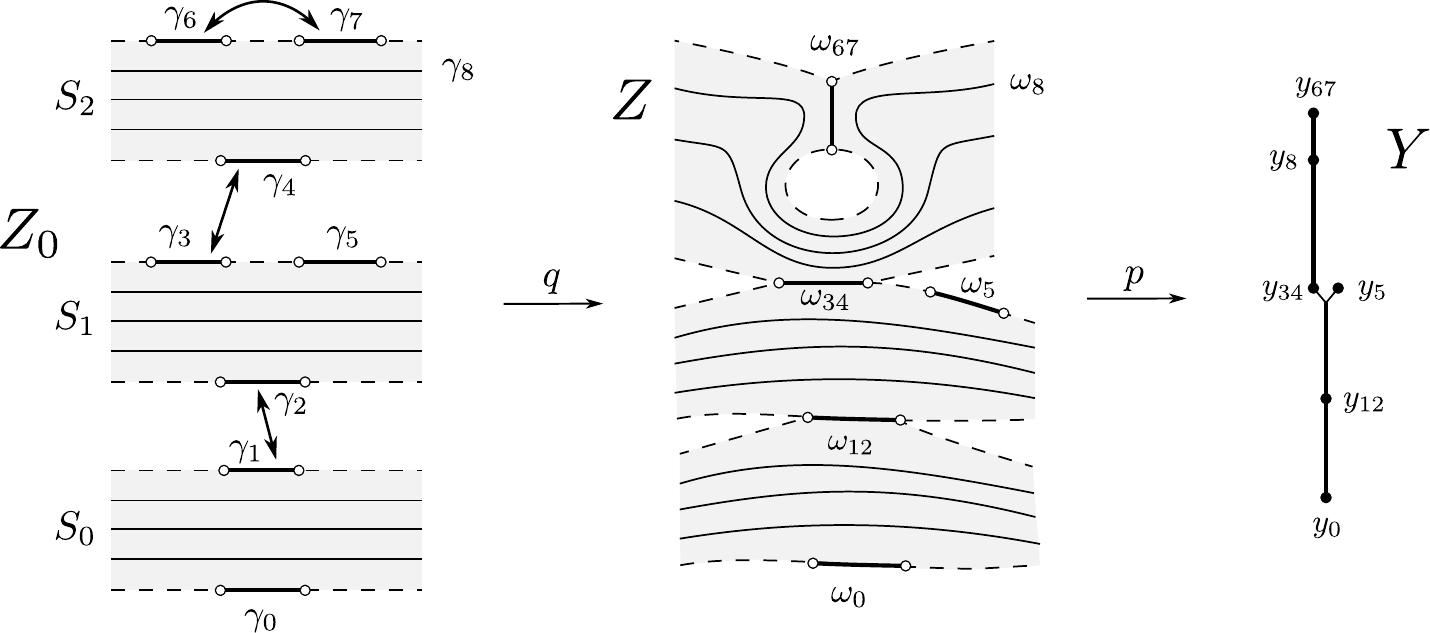}
\caption{}\label{fig:strips_example}
\end{figure}
\begin{enumerate}[wide, label={\arabic*)}]
\item
For each leaf $\gamma=\bR\times\{t\} \subset\Int{\strip_{\stInd}}$, $t\in(0,1)$,
\begin{align*}
	&\omega:=\qmap(\gamma)\in\PartitionReg, &
	&y:=\pr(\omega) \in \Int{\Yman}\setminus\specPoints,
\end{align*}
see the leaf $\omega_8$ in Figure~\ref{fig:strips_example}.

\item
Moreover, suppose $\omega\in\Partition$ is a seam, so $\qmap^{-1}(\omega) = \bdX_\bdGlueInd \cup \bdY_\bdGlueInd$, where $\bdX_\bdGlueInd \subset \bSide{\eps}{\strip_{\stInd}}$, and $\bdY_\bdGlueInd \subset \bSide{\eps'}{\strip_{\stInd'}}$ for some $\eps,\eps'\in\{0,1\}$ and $\stInd,\stInd'\in\StInd$.
Then the following characterizations hold.
\begin{enumerate}[label={\rm\alph*)}, itemsep=1ex, leftmargin=*]
\item\label{enum:regular_int_leaf}
$\omega\in\PartitionReg$ if and only if $\bdX_\bdGlueInd=\bSide{\eps}{\strip_{\stInd}}$ and $\bdY_\bdGlueInd=\bSide{\eps'}{\strip_{\stInd'}}$, see leaf $\omega_{12}$ in Figure~\ref{fig:strips_example}.

\item
$\omega\in\PartitionSing\setminus\PartitionSpec$ if and only if $\stInd=\stInd'$, $\eps=\eps'$, and $\bSide{\eps}{\strip_{\stInd}} = \bdX_\bdGlueInd \cup \bdY_\bdGlueInd$.
In other words, $\bSide{\eps}{\strip_{\stInd}}$ consists of two boundary intervals, and $\omega$ is obtained by gluing those inervals.
In this case $y:=\pr(\omega)\in\partial\Yman \setminus\specPoints$, so $\Yman$ is \myemph{Hausdorff} at $y$, see the leaf $\omega_{67}$ in Figure~\ref{fig:strips_example}.

\item
In all other cases, $\omega\in\PartitionSpec$, and $\hcl{\omega}$ contains all leaves from
\[ \qmap(\bSide{\eps}{\strip_{\stInd}}\cup\bSide{\eps'}{\strip_{\stInd'}}),\]
see leaves $\omega_{34}$ and $\omega_{5}$ in Figure~\ref{fig:strips_example}.
\end{enumerate}
\end{enumerate}
\end{subremark}

Notice that in the case~\ref{enum:regular_int_leaf} one can replace $\strip_{\stInd} \cup \strip_{\stInd'}$ with a single strip, and get another striped atlas for $(\stripSurf,\Partition)$ having one less strip.
This leads to the following notion.

A striped atlas will be called \myemph{reduced} whenever
\begin{equation}\label{equ:reduced_atlas}
\qmap(\partial\preStripSurf) = \partial\stripSurf \cup \PartitionSing.
\end{equation}

The following statement is easy and we leave it for the reader.
\begin{sublemma}[{cf.~\cite[Lemma~2.1]{MaksymenkoPolulyakh:MFAT:2016}}]\label{lm:char_leaves}
Let $\qmap:\preStripSurf \to\stripSurf$ be a striped atlas, with $\Partition$ the corresponding canonical foliation on $\stripSurf$, $\pr:\stripSurf\to\Yman$ the projection to the space of leaves, and $\specPoints$ the set of special points of $\Yman$.
Then
\[ \pr(\hcl{\omega}) = \hcl{\pr(\omega)} \]
for every leaf $\omega\in\Partition$.
Moreover,
\begin{align*}
&\pr(\PartitionSpec) = \specPoints, &
&\pr(\PartitionSing \setminus \PartitionSpec) \subset \partial\Yman\setminus\specPoints, &
&\pr(\partial\stripSurf\cup\PartitionSing) = \partial\Yman\cup\specPoints.
\end{align*}
If $\qmap$ is \myemph{reduced}, then for every connected component $\dComp$ of $\Yman\setminus\bigl(\partial\Yman\cup\specPoints)$, $\qmap^{-1}(\pr^{-1}(\dComp))$ is an interior of some model strip $\strip_{\stInd}\subset\preStripSurf$.
Define the following path $\gamma_{\dComp}:(0,1) \to  \Int{\strip_{\stInd}} = \bR\times(0,1)$ by $\gamma_{\dComp}(t) = (0, t)$, $t\in(0,1)$.
Then the composition $\pr\circ\qmap\circ\gamma_{\dComp}:(0,1) \to \dComp$ is a homeomorphism.
\qed
\end{sublemma}

\begin{subtheorem}[{\cite[Theorem~4.4]{MaksymenkoPolulyakh:PIGC:2017}}]\label{th:char_strip_surf}
Let $(\stripSurf,\Partition)$ be a foliated surface from class $\mathcal{F}$.
Then the following conditions are equivalent:
\begin{enumerate}[label={\rm(A\arabic*)}, start=4]
\item
the family $\PartitionSing$ of all singular leaves of $\Partition$ is locally finite;
\item\label{enum:A5}
$(\stripSurf,\Partition)$ is a striped surface.
\end{enumerate}
\end{subtheorem}

\begin{subtheorem}[{\cite[Theorem~3.7]{MaksymenkoPolulyakh:PGC:2015}}]\label{th:reduced_striped_surf}
Every connected striped surface $\stripSurf$ with countable base is foliated homeomorphic either to
\begin{enumerate}[label={\rm(\arabic*)}]
\item\label{enum:reduced_strip_surf:cyl}
an open cylinder $\Cman$ or an open M\"obius band $\Mman$ foliated by fibers of one-dimensional vector bundles $\Cman\to S^1$ and $\Mman\to S^1$ over the circle, or to
\item\label{enum:reduced_strip_surf:reduc}
a striped surface with reduced atlas.
\end{enumerate}
\end{subtheorem}

Notice that in the case~\ref{enum:reduced_strip_surf:cyl} all leaves of $\Partition$ are regular.
Also in the case~\ref{enum:reduced_strip_surf:reduc} the seams are precisely the singular leaves contained in $\Int{\stripSurf}$ and thus we have a canonical decomposition of $\stripSurf$ into strips.

\subsection{Characterization of the homeotopy group $\HZS$ for striped surfaces}
Let $\qmap: \preStripSurf \to \stripSurf$ be a striped atlas, and $\dif\in\HS$ a homeomorphism such that $\dif(\qmap(\strip_{\stInd})) = \qmap(\strip_{\stInd})$ for some $\stInd \in \StInd$.
Since $\qmap$ is injective on the interior of $\strip_{\stInd}$, as well as on each boundary component of $\strip_{\stInd}$, it follows that there exists a unique foliated homeomorphism $\predif:\strip_{\stInd} \to \strip_{\stInd}$ making commutative the following diagram:
\begin{equation}\label{equ:g_lift}
\begin{aligned}
\xymatrix{
\strip_{\stInd} \ar[r]^-{\qmap} \ar[d]_{\predif} & \qmap(\strip_{\stInd}) \ar[d]^-{\dif} \\
\strip_{\stInd} \ar[r]^-{\qmap}                  & \qmap(\strip_{\stInd})
}
\end{aligned}
\end{equation}
Moreover, see~\cite[Lemma~4.1]{MaksymenkoPolulyakh:PGC:2015}, $\predif(x, y) = \bigl(\flambda(x, y), \fmu(y) \bigr)$, where
\begin{itemize}[leftmargin=*]
\item $\fmu:[0,1]\to[0,1]$ is a homeomorphism;
\item for each $y\in(0,1)$ the map $\flambda_{y}:\bR\to\bR$, $\flambda_{y}(x) = \flambda(x,y)$, $x\in\bR$, is a homeomorphism;
\item for each $y\in\{0,1\}$ the map $\flambda_{y}:\bSide{y}{\strip_{\stInd}} \to \bSide{\fmu(y)}{\strip_{\stInd}}$, $\flambda_{y}(x) = \flambda(x,y)$, is a \myemph{monotone} homeomorphism between linearly ordered subsets (being unions of mutually disjoint open intervals) $\bSide{y}{\strip_{\stInd}}$ and $\bSide{\fmu(y)}{\strip_{\stInd}}$.
\end{itemize}

\begin{subtheorem}[{\cite[Theorem~4.4]{MaksymenkoPolulyakh:PGC:2015}}]\label{th:charact_HidZDelta}
Let $\qmap:\preStripSurf\to\stripSurf$ be a reduced striped atlas such that the family $\PartitionSpec$ of all special leaves of the canonical foliation $\Partition$ is locally finite.
Let also $\dif\in\HS$.
Then $\dif\in\HZS$ iff the following conditions hold:
\begin{enumerate}[label={\rm(Z\Alph*)}, leftmargin=*, itemsep=1ex]
\item\label{enum:charact_HidZP:1}
$\dif$ preserves each leaf $\omega \in \partial\stripSurf \cup \PartitionSing$ with its orientation;
\item\label{enum:charact_HidZP:2}
$\dif(\qmap(\strip_{\stInd})) = \qmap(\strip_{\stInd})$ for all $\stInd\in\StInd$, and if $\predif=(\flambda,\fmu):\strip_{\stInd} \to \strip_{\stInd}$ is a unique lifting of $\dif|_{\strip_{\stInd}}$, see~\eqref{equ:g_lift}, then $\fmu$ and each $\flambda_{y}$, $y\in[0,1]$, is strictly increasing.
\end{enumerate}
\end{subtheorem}

\section{Proof of Theorem~\ref{th:ker_ahom0}}\label{proof:th:ker_ahom0}
Let $(\stripSurf,\Partition)$ be a connected foliated surface from class $\mathcal{F}$ satisfying condition~\ref{enum:A4}.
By Theorem~\ref{th:char_strip_surf}, the latter condition is equivalent to the assumption~\ref{enum:A5} that $(\stripSurf,\Partition)$ is a striped surface, \ie there exists a striped atlas $\qmap:\preStripSurf\to\stripSurf$ such that $\Partition$ is its canonical foliation.
We need to prove that $\ahom(\HZS)= \HZY$ and that the kernel of the induced homomorphism $\ahom_0:\pi_0\HZS\to\HZY$ is either trivial or isomorphic to $\bZ_2$.

First suppose that $\stripSurf$ is connected.
Then due to Theorem~\ref{th:reduced_striped_surf} we have the following three cases.

1) $\stripSurf=\bR\times\Circle$ is an open cylinder, $\Yman=\Circle$ is the circle, $\pr:\bR\times \Circle\to \Circle$ is given by $\pr(x,y) = y$, and $\Partition$ consists of lines $\bR\times\{y\}$, $y\in\Circle$.
Then every $\dif\in\HS$ is given by $\dif(x,y) = (\flambda_{\dif}(x,y), \fmu_{\dif}(y))$, where $\fmu_{\dif}:\Circle\to\Circle$ is some homeomorphism, and $\flambda_{\dif}:\bR\times\Circle\to\bR$ is a continuous map such that the correspondence $x\mapsto\flambda_{\dif}(x,y)$ is a homeomorphism $\bR\to\bR$ which preserves or reverses orientation mutually for all $y\in\Circle$.
In particular, we have a homomorphism $\ori:\HS \to \bZ_2\oplus\bZ_2$, $\ori(\dif) = (\ori(\alpha_{\dif}), \ori(\beta_{\dif}))$, where $\ori(\beta_{\dif})=0$ if $\beta_{\dif}$ preserves orientation, $\ori(\beta_{\dif})=1$ otherwise, and similarly for $\alpha_{\dif}$.

Evidently $\ahom(\dif) = \fmu_{\dif}$.
This also implies that if $\fmu\in\HY$ and $\dif\in\HS$ is given by $\dif(x,y) = (x, \fmu(y))$, then $\ahom(\dif)=\fmu$.
Hence $\ahom$ is surjective.

It is well-known that $\HZY$ consists of homeomorphisms preserving orientation of $\Yman$.
Also, see \eg \cite[Examples~6.5]{MaksymenkoPolulyakhSoroka:PICG:2016}, one easily checks that $\dif\in\HZS$ iff $\fmu_{\dif}$ and the correspondences $x\mapsto\flambda_{\dif}(x,y)$ preserve orientations.
In particular, we get the following commutative diagram:
\[
\xymatrix{
\HZS \ar@{^(->}[r] \ar[d]_-{\ahom} &
\HS \ar[d]_-{\ahom:\dif\mapsto\beta_{\dif}} \ar@{->>}[r]^-{\ori} &
\bZ_2 \oplus \bZ_2 \ar[d]_-{(a,b)\mapsto b} \ar[r]^-{\cong} &
\pi_0\HS \ar[d]^{\ahom_0} \\
\HZY \ar@{^(->}[r] &
\HY  \ar@{->>}[r]^-{\ori} &
\bZ_2 \ar[r]^-{\cong} &
\pi_0\HY
}
\]
This implies Theorem~\ref{th:ker_ahom0} for this case.

2) The proof of Theorem~\ref{th:ker_ahom0} for the case when $\stripSurf$ is an open M\"obius band and $\pr:\stripSurf\to\Circle$ is a locally trivial fibration with fiber $\bR$ is similar, and we leave if to the reader.

3) In all other cases due to Theorem~\ref{th:reduced_striped_surf} one can assume that $\qmap$ is a reduced striped atlas.

a) First we prove that $\ahom(\HZS) \subset \HZY$ which will imply that $\ahom$ induces a homomorphism $\ahom_0:\pi_0 \HZS \to \pi_0\HZY$.

Let $\dif\in\HZS$ and $\kdif=\ahom(\kdif)$.
Thus $\dif$ satisfies conditions~\ref{enum:charact_HidZP:1} and~\ref{enum:charact_HidZP:2} of Theorem~\ref{th:charact_HidZDelta}, and we should check that $\kdif\in\HZY$, \ie it satisfies conditions~\ref{enum:cond_Hid:z} and~\ref{enum:cond_Hid:C} of Lemma~\ref{lm:spec_pt_manif}.

Due to~\ref{enum:charact_HidZP:1}, $\dif$ preserves each leaf $\omega\subset\partial\stripSurf \cup \PartitionSing$, whence $\kdif$ fixes each point $y\in \partial\Yman\cup\specPoints = \pr(\partial\stripSurf \cup \PartitionSing)$, see Lemma~\ref{lm:char_leaves}, \ie $\kdif$ satisfies condition~\ref{enum:cond_Hid:z}.

Furthermore, let $\dComp$ be a connected component of $\Yman\setminus(\partial\Yman\cup\specPoints)$.
Then by Lemma~\ref{lm:char_leaves}, $\qmap^{-1}(\pr^{-1}(\dComp))$ is the interior of some model strip $\strip_{\stInd}$.
Moreover, by property~\ref{enum:charact_HidZP:2}, $\dif(\qmap(\strip_{\stInd}))=\qmap(\strip_{\stInd})$, and $\dif$ lifts to a homeomorphism $\dif_0=(\fmu,\flambda):\strip_{\stInd}\to\strip_{\stInd}$ in which $\fmu$ is strictly increasing.
But the latter is equivalent to the assumption that $\kdif$ preserves orientation of $\dComp$, \ie satisfies condition~\ref{enum:cond_Hid:C}.

\smallskip

b) Let us show that $\ahom(\HZS)=\HZY$.

Let $\kdif\in\HZY$.
We should find $\dif\in\HZS$ such that $\ahom(\dif)=\kdif$.
In fact we will define a homeomorphism $\dif_0:\preStripSurf\to\preStripSurf$ which will induce the desired $\dif$.

Let $\dComp$ be a connected component of $\Yman\setminus\Tman$, so $r^{-1}(\dComp)$ is the interior of some model strip $\strip_{\stInd}$.
Then by Lemma~\ref{lm:char_leaves}, the map $\gamma:(0,1)\to\dComp$, $\gamma(y) = r(0,y)$, $y\in(0,1)$, is a homeomorphism.
Since by~\ref{enum:cond_Hid:C}, $\kdif(\dComp)=\dComp$ and $\kdif$ preserves its orientation, the map
\[ \fmu = \gamma^{-1}\circ\kdif\circ\gamma:(0,1)\to(0,1) \]
is an orientation preserving homeomorphism of $(0,1)$.
Therefore $\fmu$ extends to a self-homeomorphism $\fmu:[0,1]\to[0,1]$ by setting $\fmu(0)=0$ and $\fmu(1)=1$, and we define $\dif_0:\strip_{\stInd}\to\strip_{\stInd}$ by $\dif_0(x,y) = (x,\fmu(y))$, so it is also fixed on $\strip_{\stInd}$.

It now follows that $\dif_0$ induces a homeomorphism $\dif\in\HZS$ fixed on $\partial\stripSurf\cup\PartitionSing$ and such that $\dif_0\circ\qmap=\qmap\circ\dif$.
Then $\dif$ in turn induces a homeomorphism $\kdif' = \ahom(\dif):\Yman\to\Yman$ fixed on $\pr(\partial\stripSurf\cup\PartitionSing) = \partial\Yman\cup\specPoints$ and coinciding with $\kdif$ on the complement $\Yman\setminus(\partial\Yman\cup\specPoints)$.
But due to~\ref{enum:cond_Hid:z}, $\kdif$ is also fixed on $\partial\Yman\cup\specPoints$, whence $\kdif=\kdif'=\ahom(\dif)$.

\smallskip

c) Denote $\invHZY := \ahom^{-1}(\HZY)$.
Then
\begin{multline*}
\ker\Bigl[
    \ahom_0: \pi_0\HS \to \pi_0\HY
\Bigr] =
\ker\left[
    \ahom_0: \frac{\HS}{\HZS} \to \frac{\HY}{\HZY}
\right] \cong \\ \cong
\frac{\ahom^{-1}(\HZY)}{\HZS} = \invHZY / \HZS.
\end{multline*}
Hence it remains show that $\invHZY / \HZS$ is either trivial or isomorphic with $\bZ_2$.

Let $\dif\in\invHZY$, $\kdif=\ahom(\dif)\in\HZY$, and $\dif_0:\preStripSurf\to\preStripSurf$ be its unique lifting.
Then $\dif_0$ preserves every boundary component of $\preStripSurf$ and $\dif_0(\strip_{\stInd})=\strip_{\stInd}$ for all $\stInd\in\StInd$.
Moreover, $\dif_0:\Int{\strip_{\stInd}}=\bR\times(0,1)\to\bR\times(0,1)=\Int{\strip_{\stInd}}$ is given by $\dif_0(x,y) = (\flambda^{\stInd}(x,y), \fmu^{\stInd}(y))$, where $\fmu^{\stInd}$ is increasing, while $\flambda^{\stInd}_{y}$ is either increasing or decreasing mutually for all $y\in(0,1)$.

Therefore one can associate to each $\stInd$ an element $\ori_{\dif}(\stInd)\in\bZ_2$ so that $\ori_{\dif}(\stInd)=0$ (resp.\ $\ori_{\dif}(\stInd)=1$) if all $\flambda^{\stInd}_{y}$ are increasing (resp.\ decreasing).

Note that if two strips $\strip_{\stInd}$ and $\strip_{\stInd'}$ are glued along some of their boundary components, then $\ori_{\dif}(\stInd)=\ori_{\dif}(\stInd')$, whence $\ori_{\dif}$ is constant on the connected components of $\stripSurf$.

But $\stripSurf$ is connected, whence $\ori_{\dif}$ is constant.
Therefore, one can associate to each $\dif\in\invHZY$ a well-defined element $\ori_{\dif}\in\bZ_2$.
It is easy to checks that the correspondence $\dif\mapsto\ori_{\dif}$ is a homomorphism $\ori:\invHZY\to\bZ_2$.
Finally, $\dif\in\ker(\ori)$ iff $\dif$ satisfies conditions~\ref{enum:charact_HidZP:1} and~\ref{enum:charact_HidZP:2}, \ie $\dif\in\HZS$.

Therefore $\ker(\ahom_0) = \invHZY / \HZS$ is either trivial or $\bZ_2$.
This completes Theorem~\ref{th:ker_ahom0}.

\bibliographystyle{plain}
\bibliography{striped_surfaces_spaces_of_leaves}

\begin{thebibliography}{10}

\bibitem{Crowell:ProcAMS:1963}
R.~H. Crowell.
\newblock Invertible isotopies.
\newblock {\em Proc. Amer. Math. Soc.}, 14:658--664, 1963.

\bibitem{Engelking:GenTop}
R.~Engelking.
\newblock {\em General topology}, volume~6 of {\em Sigma Series in Pure
  Mathematics}.
\newblock Heldermann Verlag, Berlin, second edition, 1989.

\bibitem{Epstein:AM:1966}
D.~B.~A. Epstein.
\newblock Curves on {$2$}-manifolds and isotopies.
\newblock {\em Acta Math.}, 115:83--107, 1966.

\bibitem{Feustel:ProcAMS:1966}
C.~D. Feustel.
\newblock Homotopic arcs are isotopic.
\newblock {\em Proc. Amer. Math. Soc.}, 17:891--896, 1966.

\bibitem{Gluck:BAMS:1963}
H.~Gluck.
\newblock Restriction of isotopies.
\newblock {\em Bull. Amer. Math. Soc.}, 69:78--82, 1963.

\bibitem{GodbillonReeb:EM:1966}
C.~Godbillon and G.~Reeb.
\newblock Fibr\'es sur le branchement simple.
\newblock {\em Enseignement Math. (2)}, 12:277--287, 1966.

\bibitem{Godbillon:EM:1972}
Claude Godbillon.
\newblock Fibr\'{e}s en droites et feuilletages du plan.
\newblock {\em Enseign. Math. (2)}, 18:213--224 (1973), 1972.

\bibitem{HaefligerReeb:EM:1957}
A.~Haefliger and G.~Reeb.
\newblock Vari\'et\'es (non s\'epar\'ees) \`a une dimension et structures
  feuillet\'ees du plan.
\newblock {\em Enseignement Math. (2)}, 3:107--125, 1957.

\bibitem{HeathZhao:TopApp:1997}
P.~R. Heath and X.~Zhao.
\newblock Nielsen numbers for based maps, and for noncompact spaces.
\newblock {\em Topology Appl.}, 79(2):101--119, 1997.

\bibitem{Hirsch:DiffTop}
M.~W. Hirsch.
\newblock {\em Differential topology}, volume~33 of {\em Graduate Texts in
  Mathematics}.
\newblock Springer-Verlag, New York, 1994.
\newblock Corrected reprint of the 1976 original.

\bibitem{HuschRushing:MichJM:1969}
L.~S. Husch and T.~B. Rushing.
\newblock Restrictions of isotopies and concordances.
\newblock {\em Michigan Math. J.}, 16:303--307, 1969.

\bibitem{Kaplan:DJM:1940}
W.~Kaplan.
\newblock Regular curve-families filling the plane, {I}.
\newblock {\em Duke Math. J.}, 7:154--185, 1940.

\bibitem{Kaplan:DJM:1941}
W.~Kaplan.
\newblock Regular curve-families filling the plane, {II}.
\newblock {\em Duke Math J.}, 8:11--46, 1941.

\bibitem{KhokhliukMaksymenko:PIGC:2020}
O.~Khokhliuk and S.~Maksymenko.
\newblock Smooth approximations and their applications to homotopy types.
\newblock {\em Proc. Int. Geom. Cent.}, 13(2):68--108, 2020.

\bibitem{MaksymenkoNikitchenko:2021}
S.~Maksymenko and O.~Nikitchenko.
\newblock Fundamental groupoids and homotopy types of non-compact surfaces.
\newblock In {\em Proceedings of the {W}inter {S}chool \& {W}orkshop {W}isla
  20-21: {G}roups, invariants, integrals, and moving frames.} 2022.
\newblock To appear.

\bibitem{MaksymenkoPolulyakh:PGC:2015}
S.~Maksymenko and E.~Polulyakh.
\newblock Foliations with non-compact leaves on surfaces.
\newblock {\em Proceedings of Geometric Center}, 8(3--4):17--30, 2015.

\bibitem{MaksymenkoPolulyakh:MFAT:2016}
S.~Maksymenko and E.~Polulyakh.
\newblock Foliations with all non-­closed leaves on non-­compact surfaces.
\newblock {\em Methods Funct. Anal. Topology}, 22(3):266--282, 2016.

\bibitem{MaksymenkoPolulyakh:PGC:2016}
S.~Maksymenko and E.~Polulyakh.
\newblock One-dimensional foliations on topological manifolds.
\newblock {\em Proceedings of Geometric Center}, 9(2):1--23, 2016.

\bibitem{MaksymenkoPolulyakh:PGC:2017}
S.~Maksymenko and E.~Polulyakh.
\newblock Characterization of striped surfaces.
\newblock {\em Proceedings of the International Geometry Center}, 10(2):24--38,
  2017.

\bibitem{MaksymenkoPolulyakh:PIGC:2017}
S.~Maksymenko and E.~Polulyakh.
\newblock Characterization of striped surfaces.
\newblock {\em Proc. Int. Geom. Cent.}, 10(2):24--38, 2017.

\bibitem{MaksymenkoPolulyakhSoroka:PICG:2016}
S.~Maksymenko, E.~Polulyakh, and Yu. Soroka.
\newblock Homeotopy groups of one-dimensional foliations on surfaces.
\newblock {\em Proceedings of the International Geometry Center}, 10(1):22--46,
  2017.

\bibitem{MartinRolfsen:ProcAMS:1968}
J.~Martin and D.~Rolfsen.
\newblock Homotopic arcs are isotopic.
\newblock {\em Proc. Amer. Math. Soc.}, 19:1290--1292, 1968.

\bibitem{Soroka:MFAT:2016}
Yu. Soroka.
\newblock Homeotopy groups of rooted tree like non-singular foliations on the
  plane.
\newblock {\em Methods Funct. Anal. Topology}, 22(3):283--294, 2016.

\bibitem{Soroka:UMJ:2017}
Yu. Soroka.
\newblock Homeotopy groups of nonsingular foliations of the plane.
\newblock {\em Ukra\"{\i}n. Mat. Zh.}, 69(7):1000--1008, 2017.

\end{thebibliography}

\end{document}